\documentclass[12pt]{amsart}
\usepackage{amssymb,amsmath}
\def\H {{\mathcal H}}

\def\V {{\mathcal V}}

\def\M {{\mathcal M}}

\def\A {{\mathbb A}}

\def\R {\mathbb{R}}
\def\N {\mathbb{N}}

\def\D {{\mathfrak D}}

\def\eps{\varepsilon}
\def\e{{\rm e}}

\def\d{{\rm d}}
\def\ddt{\frac{\d}{\d t}}

\def\i{{\rm i}}
\def \l {\langle}
\def \r {\rangle}

\def \and {{\qquad\text{and}\qquad}}
\def \imply {{\quad\Rightarrow\quad}}

\newtheorem{proposition}{Proposition}
\newtheorem{theorem}[proposition]{Theorem}

\newtheorem{lemma}[proposition]{Lemma}
\theoremstyle{definition}

\newtheorem*{remark}{Remark}
\numberwithin{equation}{section}
\def \au {\rm}
\def \ti {\it}
\def \jou {\rm}
\def \bk {\it}
\def \no#1#2#3 {{\bf #1} (#3), #2.}
\def \eds#1#2#3 {#1, #2, #3.}

\title[Timoshenko systems with Gurtin-Pipkin thermal law]
{On the stability of Timoshenko systems\\
with Gurtin-Pipkin thermal law}

\author[F. Dell'Oro and V. Pata]
{Filippo Dell'Oro and Vittorino Pata}

\address{Politecnico di Milano - Dipartimento di Matematica ``F.\ Brioschi''
\newline\indent
Via Bonardi 9, 20133 Milano, Italy}
\email{filippo.delloro@mail.polimi.it}
\email{vittorino.pata@polimi.it}


\begin{document}

\begin{abstract}
We analyze the differential system
$$
\begin{cases}
\rho_1 \varphi_{tt} -\kappa(\varphi_x +\psi)_x = 0\\
\noalign{\vskip2mm}
\rho_2 \psi_{tt} -b\psi_{xx} +\kappa(\varphi_x +\psi) +\delta\theta_x= 0\\
\noalign{\vskip1mm}
\displaystyle
\rho_3 \theta_t  - \frac{1}{\beta}\int_0^\infty g(s)\theta_{xx}(t-s)\,\d s + \delta\psi_{tx}=0
\end{cases}
$$
describing a Timoshenko beam coupled with a temperature evolution of Gurtin-Pipkin type.
A necessary and sufficient condition for exponential stability is established in terms
of the structural parameters of the equations. In particular,
we generalize previously known results on the Fourier-Timoshenko and the Cattaneo-Timoshenko beam
models.
\end{abstract}

\maketitle

\section{Introduction}

\noindent
Given a real interval ${\mathfrak I}=[0,\ell]$,
we consider the thermoelastic beam model of Timoshenko type \cite{TIM}
\begin{equation}
\label{TIMO0}
\begin{cases}
\rho_1 \varphi_{tt} -\kappa(\varphi_x +\psi)_x = 0,\\
\rho_2 \psi_{tt} -b\psi_{xx} +\kappa(\varphi_x +\psi) +\delta\theta_x= 0,\\
\rho_3 \theta_t + q_x +\delta\psi_{tx}=0,
\end{cases}
\end{equation}
where the unknown variables
$$\varphi,\psi,\theta,q: \,(x,t)\in {\mathfrak I}\times [0,\infty) \mapsto \R$$
represent the transverse displacement of a beam with reference configuration ${\mathfrak I}$,
the rotation angle of a filament, the relative temperature (i.e.\ the temperature variation
field from an equilibrium reference value) and the heat flux vector, respectively.
Here, $\rho_1,\rho_2,\rho_3$ as well as $\kappa,b,\delta$ are
strictly positive fixed constants. The system is complemented with
the Dirichlet boundary conditions for $\varphi$ and $\theta$
$$
\varphi(0,t)=\varphi(\ell,t)=\theta(0,t)=\theta(\ell,t)=0,
$$
and the Neumann one for $\psi$
$$
\psi_x(0,t)=\psi_x(\ell,t)=0.
$$
Such conditions, commonly adopted in the literature,
seem to be the most feasible from a physical viewpoint.
To complete the picture,
a further relation is needed:
the so-called constitutive law for the heat flux,
establishing a link between $q$ and $\theta$.
This is what really characterizes the dynamics,
since no mechanical dissipation is present in the system,
and any possible loss of energy can be due only to thermal effects.

\subsection{The Fourier thermal law}
A first choice is to assume the classical Fourier law of heat conduction
\begin{equation}
\label{FOURIER}
\beta q + \theta_x = 0,
\end{equation}
where $\beta>0$ is a fixed constant.
In which case, the third equation of \eqref{TIMO0} becomes
$$
\rho_3 \theta_t - \frac1\beta \theta_{xx} +\delta\psi_{tx}=0.
$$
The exponential stability of the resulting Timoshenko-Fourier system has been analyzed in~\cite{RR}.
There, the authors introduce the so-called stability number\footnote{The notion
of stability number is actually introduced in
the subsequent paper \cite{SJR}, defined there as $\chi={\kappa}/{\rho_1}-{b}/{\rho_2}$.
The difference is clearly irrelevant with respect to the relation $\chi=0$.
The motivation of our choice
of $\chi$ is to
render more direct the comparison
with the Cattaneo law.}
$$
\chi=\frac{\rho_1}{\kappa}-\frac{\rho_2}{b},
$$
representing the difference of the inverses of the propagation speeds.
The main result of~\cite{RR} reads as follows:
the contraction semigroup generated by \eqref{TIMO0}-\eqref{FOURIER}
acting on the triplet $(\varphi,\psi,\theta)$ is exponentially stable
(in the natural weak energy space) if and only if $\chi=0$.

\subsection{The Cattaneo thermal law}
The drawback of the Fourier law lies in the physical paradox of infinite propagation speed
of (thermal) signals, a typical side-effect
of parabolicity. A different model, removing this paradox, is the
Cattaneo law~\cite{CAT},
namely, the differential perturbation
of~\eqref{FOURIER}
\begin{equation}
\label{MC}
\tau q_t + \beta q + \theta_x = 0,
\end{equation}
for $\tau>0$ small.
A natural question is whether the semigroup generated by \eqref{TIMO0} coupled with~\eqref{MC},
now acting on the state variable $(\varphi,\psi,\theta,q)$, remains exponentially stable
within the condition $\chi=0$ above. As shown in~\cite{SR}, the answer is negative:
exponential stability can never occur when $\chi=0$.
More recently, in~\cite{SJR} a new stability number is introduced
in order to deal with the Timoshenko-Cattaneo system, that is,\footnote{The value of $\chi_\tau$
in \cite{SJR} differs from ours for a multiplicative constant.}
$$
\chi_\tau =  \bigg[\frac{\rho_1}{\rho_3 \kappa}-\tau \bigg]
\bigg[\frac{\rho_1}{\kappa}-\frac{\rho_2}{b}\bigg] - \tau \frac{\,\rho_1\delta^2}{\rho_3\kappa b}.
$$
The system is shown to be exponentially stable if and only if $\chi_\tau=0$.
Quite interestingly, the Fourier case is fully recovered in the limit $\tau\to 0$,
when \eqref{MC} collapses into~\eqref{FOURIER}. Indeed, for $\tau=0$ the equality
$$\chi_0=\frac{\rho_1}{\rho_3 \kappa}\,\chi
$$
holds, which tells at once that
$$\chi_0=0\quad\Leftrightarrow\quad\chi=0.$$
It is worth mentioning that
the proof of exponential stability in~\cite{SJR} is carried out via linear semigroup techniques, whereas
the analogous result of \cite{RR} for the Fourier case is obtained by constructing explicit energy functionals.

\subsection{The Gurtin-Pipkin thermal law}
The aim of the present work is studying the Timoshenko system \eqref{TIMO0}
assuming the Gurtin-Pipkin heat conduction law for the heat flux \cite{GP}.
More precisely, we consider the constitutive equation
\begin{equation}
\label{GPlaw}
\beta q(t) + \int_{0}^\infty g(s)\theta_x(t-s)\,\d s =0,
\end{equation}
where $g$, called the memory kernel, is a (bounded)
convex summable function on $[0,\infty)$ of total mass
$$
\int_0^\infty g(s)\,\d s = 1,
$$
whose properties will be specified in more detail later on.
Equation \eqref{GPlaw} can be viewed as a memory relaxation of the Fourier law \eqref{FOURIER},
inducing (similarly to the Cattaneo law) a fully hyperbolic mechanism of heat transfer.
In this perspective, it may be considered a more realistic description of physical reality.
Accordingly, system \eqref{TIMO0} turns into
\begin{equation}
\label{sys}
\begin{cases}
\rho_1 \varphi_{tt} -\kappa(\varphi_x +\psi)_x = 0,\\
\noalign{\vskip2mm}
\rho_2 \psi_{tt} -b\psi_{xx} +\kappa(\varphi_x +\psi) +\delta\theta_x= 0,\\
\noalign{\vskip1mm}
\displaystyle
\rho_3 \theta_t  - \frac{1}{\beta}\int_0^\infty g(s)\theta_{xx}(t-s)\,\d s + \delta\psi_{tx}=0.
\end{cases}
\end{equation}
Rephrasing system \eqref{sys} within the history framework of Dafermos \cite{DAF},
we construct a contraction semigroup $S(t)$ of solutions acting on a suitable Hilbert space $\H$,
accounting for the presence of the memory. Then,
introducing the stability number
$$
\chi_g =  \bigg[\frac{\rho_1}{\rho_3 \kappa}-\frac{\beta}{g(0)} \bigg]
\bigg[\frac{\rho_1}{\kappa}-\frac{\rho_2}{b}\bigg]
-\frac{\beta}{g(0)}\frac{\,\rho_1\delta^2}{\rho_3\kappa b},
$$
our main theorem can be stated as follows.

\begin{theorem}
\label{MAIN}
The semigroup $S(t)$ is exponentially stable\footnote{We recall
that $S(t)$ is said to be exponentially stable on $\H$ if there are $\omega>0$ and $C\geq 1$
such that
$$\|S(t)z\|_{\H}\leq C\e^{-\omega t}\|z\|_\H,\quad\forall z\in\H.$$} if and only if $\chi_g=0$.
\end{theorem}

As we will see in the next section, Theorem~\ref{MAIN} actually subsumes
and generalizes all the previously known results on the exponential decay
properties of the thermoelastic Timoshenko system~\eqref{TIMO0}.

\subsection*{Plan of the paper}
In Section~\ref{Comparison} we compare Timoshenko systems
of the form~\eqref{TIMO0}
subject to different laws of heat conduction,
viewed as particular instances of~\eqref{sys} for suitable choices of the memory kernel.
The comparison with the Timoshenko-Cattaneo system of~\cite{SR,SJR},
only formal at this stage, is rendered
rigorous in the final Section~\ref{GoodComp}.
After introducing some notation (Section~\ref{FSN}),
in Section~\ref{Semigroup} we define the semigroup
$S(t)$ describing the solutions to~\eqref{sys}.
The subsequent three sections are devoted to the proof of Theorem~\ref{MAIN}.
Firstly, we introduce some auxiliary functionals (Section~\ref{Aux}), needed in the proof
of the sufficiency part of the theorem carried out in Section~\ref{Suf}. The necessity
of the condition $\chi_g=0$ in order for exponential stability to occur is proved in
Section~\ref{Nec}.

\section{Comparison with Earlier Results}
\label{Comparison}

\subsection{The Fourier case}
The Fourier law \eqref{FOURIER} can be seen as a (singular) limit of the Gurtin-Pipkin law~\eqref{GPlaw}.
Indeed, defining the $\eps$-scaling of the memory kernel $g$ by
$$g_\eps(s)=\frac1\eps g\Big(\frac{s}\eps\Big),\quad\eps>0,$$
we consider in place of the original \eqref{GPlaw} the constitutive equation
\begin{equation}
\label{GPlawS}
\beta q(t) + \int_{0}^\infty g_\eps(s)\theta_x(t-s)\,\d s =0.
\end{equation}
Since $g_\eps\to \delta_0$ in the distributional sense, where $\delta_0$ denotes the Dirac mass at $0^+$,
it is clear that, in the limit $\eps\to 0$, equation~\eqref{GPlawS} reduces to
the classical constitutive law~\eqref{FOURIER}.
According to Theorem~\ref{MAIN}, exponential stability for the
Timoshenko-Gurtin-Pipkin model with memory kernel $g_\eps$ occurs if and only if
$$\chi_{g_\eps} =  \bigg[\frac{\rho_1}{\rho_3 \kappa}-\frac{\beta\eps}{g(0)}
\bigg]\bigg[\frac{\rho_1}{\kappa}-\frac{\rho_2}{b}\bigg]
-\frac{\beta\eps}{g(0)} \frac{\,\rho_1\delta^2}{\rho_3\kappa b}=0.
$$
Letting $\eps\to 0$, we recover the condition
$$\chi_{\delta_0} = \frac{\rho_1}{\rho_3 \kappa}\,\chi=0
$$
of the Fourier case. The convergence of the Timoshenko-Gurtin-Pipkin model to the
Timoshenko-Fourier one as $\eps\to 0$ can be made rigorous
within the proper functional setting, along the same lines of~\cite{Amnesia}.

\subsection{The Cattaneo case}
The Cattaneo law~\eqref{MC} can be deduced as a particular
instance of \eqref{GPlaw},
corresponding to the memory kernel
$$
g_\tau(s)=\frac{\beta}{\tau} \,\e^{-\frac{s \beta}{\tau}}.
$$
Indeed, changing the integration variable, we can write the flux vector $q$ in the form
$$q(t)=-\frac1\beta\int_{-\infty}^t g_\tau(t-s)\theta_x(s)\,\d s.
$$
Since
$$g_\tau'(s)=-\frac{\beta}\tau\,g_\tau(s),
$$
we draw the relation
$$q_t(t)
=-\frac1\beta\int_{-\infty}^t g_\tau'(t-s)\theta_x(s)\,\d s-\frac{g_\tau(0)}\beta\theta_x(t)
=\frac1\tau\int_{-\infty}^t g_\tau(t-s)\theta_x(s)\,\d s-\frac1\tau\theta_x(t),
$$
which is nothing but \eqref{MC}.
Besides, we have
the equality of the stability numbers
$$\chi_{g_\tau}=\chi_\tau.
$$

\subsection{The Coleman-Gurtin case}
A further interesting model, midway between the
Fourier and the Gurtin-Pipkin one,
is obtained by assuming the (parabolic-hyperbolic)
Coleman-Gurtin law for the heat flux, namely,
\begin{equation}
\label{CGlaw}
\beta q(t) +(1-\alpha)\theta_x(t)+\alpha \int_{0}^\infty g(s)\theta_x(t-s)\,\d s =0,\quad\alpha\in(0,1).
\end{equation}
The limit cases $\alpha=0$ and $\alpha=1$ correspond to the fully parabolic Fourier case and the fully
hyperbolic Gurtin-Pipkin case.
The corresponding Timoshenko-Coleman-Gurtin system, whose third equation now reads
$$
\rho_3 \theta_t  - \frac{1}{\beta}\bigg[(1-\alpha)\theta_{xx}
+\alpha \int_0^\infty g(s)\theta_{xx}(t-s)\,\d s\bigg] + \delta\psi_{tx}=0,
$$
generates (similarly to the Timoshenko-Gurtin-Pipkin system)
a contraction semigroup $\Sigma(t)$ on $\H$.
For this system, the following theorem holds.

\begin{theorem}
\label{MAIN-CG}
The semigroup $\Sigma(t)$ is exponentially stable if and only if $\chi=0$.
\end{theorem}

Hence, the picture is exactly the same as in the Fourier case. This, as observed in \cite{SJR},
is due to the predominant character of parabolicity.
Theorem~\ref{MAIN-CG} can be given a direct proof, following the lines of the next sections.
In fact, the situation here is much simpler, do to the presence of instantaneous dissipation
given by the term $-\theta_{xx}$
in the equation. However, it is also possible to obtain Theorem~\ref{MAIN-CG} as a byproduct
of Theorem~\ref{MAIN}. To this end, it is enough to consider the Timoshenko-Gurtin-Pipkin system
with kernel
$$g_\eps(s)=\frac{1-\alpha}\eps g\Big(\frac{s}\eps\Big)+\alpha g(s),$$
whose exponential stability takes place if and only if
$$\chi_{g_\eps} =  \bigg[\frac{\rho_1}{\rho_3 \kappa}
-\frac{\beta\eps}{(1-\alpha+\alpha\eps)g(0)} \bigg]\bigg[\frac{\rho_1}{\kappa}-\frac{\rho_2}{b}\bigg]
-\frac{\beta\eps}{(1-\alpha+\alpha\eps)g(0)} \frac{\,\rho_1\delta^2}{\rho_3\kappa b}=0.
$$
Performing the limit $\eps\to 0$, we obtain the distributional convergence
$$g_\eps\to (1-\alpha)\delta_0+\alpha g,
$$
yielding in turn
$$
\int_0^\infty g_\eps(s)\theta_{xx}(t-s)\,\d s
\to
(1-\alpha)\theta_{xx}
+\alpha \int_0^\infty g(s)\theta_{xx}(t-s)\,\d s.
$$
Accordingly, we see that
$$\chi_{(1-\alpha)\delta_0+\alpha g} = \frac{\rho_1}{\rho_3 \kappa}\,\chi,
$$
and we recover the same stability condition of the Timoshenko-Fourier system.

\subsection{Heat conduction of type III}
We finally mention another model, resulting from the constitutive law
of type III of Green-Naghdi for the heat flux \cite{GN0,GN1,GN2}
\begin{equation}
\label{TIII}
\beta q + \theta_x +d p_x= 0,\quad d>0,
\end{equation}
where
$$p(t)=p(0)+\int_0^t \theta(r)\,\d r
$$
is the so-called thermal displacement. Plugging \eqref{TIII} into \eqref{TIMO0},
one obtains
$$
\begin{cases}
\rho_1 \varphi_{tt} -\kappa(\varphi_x +\psi)_x = 0,\\
\rho_2 \psi_{tt} -b\psi_{xx} +\kappa(\varphi_x +\psi) +\delta p_{tx}= 0,\\
\rho_3 p_{tt} - \frac1\beta p_{txx} - \frac{d}\beta p_{xx}+\delta\psi_{tx}=0.
\end{cases}
$$
In \cite{MS}, the system is shown to be exponentially stable if $\chi=0$.
Again, the partial parabolicity of the model prevails, so that the exponential stability condition
is the same as in the Fourier case. Though, the constitutive law \eqref{TIII} cannot be deduced from
\eqref{GPlaw}, not even as a limiting case. Possibly, this feature may reflect
the fact that the theory
of heat conduction of type III seems to be at the limit
of thermodynamic admissibility (see the analysis of \cite{GGP}).

\section{Functional Setting and Notation}
\label{FSN}

\subsection{Assumptions on the memory kernel}
Calling
$$\mu(s)=-g'(s),$$
the {\it prime} denoting the derivative with respect to $s$,
let the following conditions hold.
\begin{itemize}
\smallskip
\item[(i)] $\mu$ is a nonnegative nonincreasing absolutely continuous function on $\R^+$
such that
$$\mu(0)=\lim_{s\to 0}\mu(s)\in (0,\infty).$$
\item[(ii)] There exists $\nu>0$ such that the differential inequality
$$
\mu'(s) + \nu \mu(s) \leq 0
$$
holds for almost every $s>0$.
\end{itemize}
\smallskip

\begin{remark}
For every $k>0$, the exponential kernel
$g(s)=k \e^{-ks}$ meets the hypotheses (i)-(ii).
\end{remark}

In particular, $\mu$ is summable on $\R^+$ with
$$\int_0^\infty \mu(s)\,\d s=g(0).
$$
Besides, the requirement that $g$ has total mass 1 translates into
$$
\int_0^\infty s\mu(s)\,\d s=1.
$$

\subsection{Functional spaces}
In what follows, $\langle\cdot,\cdot\rangle$ and $\|\cdot\|$ are the
standard inner product and norm on the Hilbert space $L^2({\mathfrak I})$.
We introduce the Hilbert subspace
$$
L^2_*({\mathfrak I})=\bigg\{ f\in L^2({\mathfrak I}) : \int_0^\ell f(x)\, \d x = 0\bigg\}
$$
of zero-mean functions, along with
the Hilbert spaces
$$H_0^1({\mathfrak I})\qquad\text{and}\qquad H^1_*({\mathfrak I}) = H^1({\mathfrak I})\cap L^2_*({\mathfrak I}),
$$
both endowed with the gradient norm, due to the Poincar\'e inequality.
We also consider the space $H^2({\mathfrak I})$ and
so-called memory space
$$
\M = L^2(\R^+; H_0^1({\mathfrak I}))
$$
of square summable $H_0^1$-valued functions on $\R^+$ with respect to the measure $\mu(s)\d s$,
endowed with the inner product
$$\langle\eta,\xi\rangle_\M=\int_0^\infty \mu(s)\langle\eta_x(s),\xi_x(s)\rangle\,\d s.
$$
The infinitesimal generator of the right-translation semigroup on $\M$ is
the linear operator
$$T\eta=-D\eta$$
with domain
$$\D(T)=\Big\{\eta\in{\M}:D\eta\in\M,\,\,
\lim_{s\to 0}\|\eta_x(s)\|=0\Big\},$$
where $D$ stands for weak derivative with respect to the internal variable $s\in\R^+$.
The phase space of our problem will be
$$
\H = H_0^1({\mathfrak I})\times L^2({\mathfrak I})\times H^1_*({\mathfrak I}) \times L^2_*({\mathfrak I}) \times L^2({\mathfrak I}) \times \M
$$
normed by
$$
\|(\varphi,\tilde\varphi,\psi,\tilde\psi,\theta,\eta)\|_\H^2
= \kappa\|\varphi_x+\psi\|^2 +\rho_1\|\tilde\varphi\|^2
+b\|\psi_x\|^2+\rho_2\|\tilde\psi\|^2 + \rho_3\|\theta\|^2 + \frac1\beta\|\eta\|^2_\M.
$$

\subsection{Basic facts on the memory space}
For every $\eta\in\D(T)$, the nonnegative functional
$$
\Gamma [\eta]=
-\int_0^{\infty}
\mu'(s)\|\eta_x(s)\|^2\,\d s
$$
is well defined, and the following identity holds (see \cite{Terreni})
\begin{equation}
\label{TTT}
2\l T \eta,\eta\r_{\M}
=-\Gamma[\eta].
\end{equation}
Moreover, in light of assumption (ii) on $\mu$,
we deduce the inequality
\begin{equation}
\label{NormaETA}
\nu\|\eta\|_{\M}^2\leq\Gamma [\eta],
\end{equation}
which will be crucial for our purposes.

\section{The Contraction Semigroup}
\label{Semigroup}

\noindent
Firstly, we introduce the auxiliary variable
$$\eta=\eta^t(x,s): \,(x,t,s)\in {\mathfrak I}\times [0,\infty) \times \R^+ \mapsto \R,$$
accounting for the integrated past history of $\theta$
and formally
defined as (see \cite{DAF,Terreni})
$$
\eta^t(x,s)=\int_{0}^s \theta(x,t-\sigma)\,\d\sigma,
$$
thus satisfying the Dirichlet boundary condition
$$
\eta^t(0,s)=\eta^t(\ell,s)=0
$$
and the further ``boundary condition"
$$\lim_{s\to 0}\eta^t(x,s)=0.$$
Hence, $\eta$ satisfies the equation
$$\eta^t_t=-\eta^t_s+\theta(t).
$$
The way to render the argument rigorous is
recasting \eqref{sys} in the history space
framework devised by C.M.\ Dafermos \cite{DAF}. This
amounts to considering
the partial differential system in the unknowns $\varphi=\varphi(t)$, $\psi=\psi(t)$,
$\theta=\theta(t)$ and $\eta=\eta^t$
\begin{align}
\label{Tuno}
&\rho_1 \varphi_{tt} -\kappa(\varphi_x +\psi)_x = 0,\\
\noalign{\vskip2,5mm}
\label{Tdue}
&\rho_2 \psi_{tt} -b\psi_{xx} +\kappa(\varphi_x +\psi) +\delta\theta_x= 0,\\
\noalign{\vskip1mm}
\displaystyle
\label{Ttre}
&\rho_3 \theta_t  - \frac{1}{\beta}\int_0^\infty \mu(s)\eta_{xx}(s)\,\d s + \delta\psi_{tx}=0,\\
\noalign{\vskip1mm}
\label{Tquattro}
&\eta_t = T\eta + \theta.
\end{align}

\begin{remark}
The analogy with the original system~\eqref{sys} is not merely  formal,
and can be made rigorous within the
proper functional setting (see \cite{Terreni} for more details).
\end{remark}

Introducing the state vector
$$z(t)=(\varphi(t),\tilde\varphi(t),\psi(t),\tilde\psi(t),\theta(t),\eta^t),$$
we view \eqref{Tuno}-\eqref{Tquattro} as the ODE in $\H$
\begin{equation}
\label{EQLIN}
\ddt z(t)=\A z(t),
\end{equation}
where the linear operator $\A$ is defined as
$$
\A
\left(\begin{matrix}
\varphi\\
\tilde\varphi\\
\psi\\
\tilde\psi\\
\theta\\
\eta
\end{matrix}
\right)
=\left(
\begin{matrix}
\tilde\varphi\\
\frac{\kappa}{\rho_1} (\varphi_x + \psi)_x\\
\tilde\psi\\
\frac{b}{\rho_2}\psi_{xx} - \frac{\kappa}{\rho_2}(\varphi_x+\psi) - \frac{\delta}{\rho_2}\theta_x\\
\noalign{\vskip1mm}
\frac{1}{\beta\rho_3}\int_0^\infty \mu(s) \eta_{xx}(s)\,\d s - \frac{\delta}{\rho_3}\tilde\psi_x\\
\noalign{\vskip.5mm}
T\eta + \theta
\end{matrix}
\right)
$$
with domain
$$
\D(\A) =\left\{ (\varphi,\tilde\varphi,\psi,\tilde\psi,\theta,\eta)\in\H \left|\,\,
\begin{matrix}
\varphi\in H^2({\mathfrak I})\\
\tilde\varphi\in H_0^1({\mathfrak I})\\
\psi_x\in H_0^1({\mathfrak I})\\
\tilde\psi \in H_*^1({\mathfrak I})\\
\theta \in H_0^1({\mathfrak I})\\
\eta \in \D(T)\\
\int_0^\infty \mu(s)\eta(s)\, \d s \in H^2({\mathfrak I})\\
\end{matrix}\right.
\right\}.
$$

\begin{theorem}
\label{Sgeneration}
The operator $\A$ is the infinitesimal generator of
a contraction semigroup
$$
S(t) = \e^{t\A}:\H\to\H.
$$
\end{theorem}

The proof of this fact is based on
the classical Lumer-Phillips theorem \cite{PAZ}, and is here omitted
(see however \cite{butanino} for an application of the Lumer-Phillips theorem to equations
with memory in the past history framework).
Thus, for every initial datum
$$z_0=(\varphi_0,\tilde\varphi_0,\psi_0,\tilde\psi_0,\theta_0,\eta_0)\in\H$$
given at time $t=0$,
the unique solution at time $t>0$ to \eqref{EQLIN} reads
$$z(t)=(\varphi(t),\varphi_t(t),\psi(t),\psi_t(t),\theta(t),\eta^t)=S(t)z_0.$$
Besides, $\eta^t$ fulfills the explicit representation
formula (see~\cite{Terreni})
$$
\eta^t(s)=\displaystyle
\begin{cases}
\int_0^s \theta(t-\sigma)\,\d \sigma & s\leq t,\\
\eta_0(s-t)+\int_0^t \theta(t-\sigma)\,\d \sigma & s>t.
\end{cases}
$$

\begin{remark}
As observed in \cite{RR}, the choice of the spaces of zero-mean functions
for the variable
$\psi$ and its derivative is consistent. Indeed, calling
$$\Theta(t)=\int_0^\ell \psi(x,t)\,\d x$$
and integrating \eqref{Tdue} on ${\mathfrak I}$ we obtain the differential equation
$$\rho_2 \ddot \Theta(t)+\kappa \Theta(t)=0.$$
Hence, if $\Theta(0)=\dot\Theta(0)=0$ it follows that $\Theta(t)\equiv 0$.
\end{remark}

\begin{remark}
For the existence of the contraction semigroup $S(t)$
the hypotheses (i)-(ii) on the kernel are overabundant.
It is actually enough to require that
$\mu$ be a (nonnull and nonnegative) nonincreasing absolutely continuous
summable function on $\R^+$, possibly unbounded in a neighborhood of zero.
\end{remark}

For any fixed initial datum $z_0\in\H$, we define
(twice) the energy as
$$
E(t)=\|S(t)z_0\|_\H^2.
$$
The natural multiplication of equation \eqref{EQLIN} by $z(t)$ in the weak energy space,
along with an exploitation of~\eqref{TTT}, provide the energy identity
\begin{equation}
\label{Eidentity}
\ddt E(t) =2\langle \A z(t),z(t)\rangle_\H
=\frac{2}{\beta}\langle T \eta^t,\eta^t\rangle_\M=-\frac{1}{\beta} \Gamma[\eta^t],
\end{equation}
valid for all $z_0\in\D(\A)$.

\smallskip
As anticipated in the Introduction, the main Theorem~\ref{MAIN}
of this paper tells that
$$\text{$S(t)$ exp.\ stable}\quad\Leftrightarrow\quad\chi_g=0.$$
The proof of the result is carried out in the next Sections~\ref{Aux}-\ref{Nec}.

\begin{remark}
We mention that an alternative approach is also possible.
Namely, to set the problem in the so-called minimal state framework \cite{FGP},
rather than in the past history one.
In which case, the necessary and sufficient condition $\chi_g=0$ of exponential decay
remains the same. In fact, one can show in general that
the exponential decay in the history space (i.e.\ what proved here) implies
the analogous decay in the minimal state space (cf.\ \cite{comarpa,FGP}).
\end{remark}

\section{Some Auxiliary Functionals}
\label{Aux}

\noindent
In this section, we define some auxiliary functionals needed in the proof of the sufficiency
part of Theorem~\ref{MAIN}. As customary, it is understood that we work with
(regular) solutions arising from
initial data belonging to the domain of the operator $\A$. Along the section, $C\geq0$ will denote a {\it generic}
constant depending only on the structural quantities of the problem.
Besides, we will tacitly use several times the H\"older, Young and Poincar\'e inequalities.
In particular we will exploit the inequality
$$\int_0^\infty \mu(s) \|\eta_x(s)\|\,\d s\leq\bigg(\int_0^\infty \mu(s)\,\d s\bigg)^\frac12
\bigg(\int_0^\infty \mu(s) \|\eta_x(s)\|^2\,\d s\bigg)^\frac12
=\sqrt{g(0)}\,\|\eta\|_\M.
$$

\subsection{The functional $\boldsymbol{I}$}
Let
$$
I(t) = -\frac{2\rho_3}{g(0)}\int_0^\infty \mu(s) \l \theta(t),\eta^t(s) \r\,\d s.
$$

\begin{lemma}
\label{Li}
For every $\eps_I>0$ small, $I$ satisfies the differential inequality
$$
\ddt I + \rho_3\|\theta\|^2 +\|\eta\|^2_\M\leq \eps_I\|\psi_t\|^2
+ \frac{c_I}{\eps_I} \Gamma[\eta]
$$
for some $c_I>0$ independent of $\eps_I$.
\end{lemma}

\begin{proof}
In light of \eqref{Ttre} and \eqref{Tquattro}, we have the identity
\begin{align*}
\ddt I + 2\rho_3\|\theta\|^2 +\|\eta\|^2_\M= &-\frac{2\rho_3}{g(0)}
\int_0^\infty \mu(s) \l T\eta(s),\theta \r\,\d s + \frac{2}{g(0)\beta}\bigg\|
\int_0^\infty \mu(s)\eta_x(s)\,\d s\bigg\|^2\\
&- \frac{2\delta}{g(0)}\int_0^\infty \mu(s) \l \eta_x(s),\psi_t \r\,\d s+\|\eta\|^2_\M.
\end{align*}
Integrating by parts in $s$, we infer that
(as shown in~\cite{Terreni}, the boundary terms vanish)
\begin{align*}
-\frac{2\rho_3}{g(0)}\int_0^\infty \mu(s) \l T\eta(s),\theta \r\,\d s &=
-\frac{2\rho_3}{g(0)}\int_0^\infty \mu'(s) \l \eta(s),\theta \r\,\d s \\
&\leq C \|\theta\|\sqrt{\Gamma[\eta]}\\
\noalign{\vskip1mm}
&\leq \rho_3 \|\theta\|^2 + C\Gamma[\eta].
\end{align*}
Thus, exploiting the inequalities
$$
\frac{2}{g(0)\beta}\bigg\|\int_0^\infty \mu(s)\eta_x(s)\,\d s\bigg\|^2 \leq C \|\eta\|_\M^2
$$
and
$$
- \frac{2\delta}{g(0)}\int_0^\infty \mu(s) \l \eta_x(s),\psi_t \r\,\d s\leq C \|\psi_t\|\|\eta\|_\M,
$$
appealing to \eqref{NormaETA} we obtain for every $\eps_I>0$ small the estimate
$$
\ddt I + \rho_3 \|\theta\|^2 +\|\eta\|^2_\M\leq C \|\psi_t\|\Gamma[\eta] + C\Gamma[\eta]
\leq \eps_I\|\psi_t\|^2 + \frac{C}{\eps_I} \Gamma[\eta],
$$
where $C$ is independent of $\eps_I$.
\end{proof}

\subsection{The functional $\boldsymbol{J}$}
Defining the primitive\footnote{In particular, $\Psi\in H_0^1({\mathfrak I})$.}
$$
\Psi(x,t) = \int_0^x \psi(y,t)\,\d y,
$$
let
$$
J(t) = -\frac{2\rho_2\rho_3}{\delta} \l \theta(t), \Psi_t(t)\r.
$$

\begin{lemma}
\label{Lgei}
For every $\eps_J>0$ small, $J$ satisfies the differential inequality
$$
\ddt J + \rho_2\|\psi_t\|^2 \leq \eps_J \big[
 \|\psi_x\|^2 + \|\varphi_x + \psi\|^2 \big] + \frac{c_J}{\eps_J}\big[\|\theta\|^2+ \Gamma[\eta]\big]
$$
for some $c_J>0$ independent of $\eps_J$.
\end{lemma}

\begin{proof}
By means of \eqref{Tdue} and \eqref{Ttre}, we get
\begin{align*}
&\ddt J + 2 \rho_2 \|\psi_t\|^2\\
&= \frac{2 \rho_2}{\beta\delta}\int_0^\infty \mu(s) \l \eta_x(s),\psi_t\r\,\d s
- \frac{2 \rho_3 b}{\delta} \l\theta,\psi_x\r + \frac{2 \rho_3\kappa}{\delta}\l \theta,\varphi + \Psi \r
+ 2\rho_3\|\theta\|^2.
\end{align*}
Estimating the terms in the right-hand side as (here we use again \eqref{NormaETA})
$$
\frac{2 \rho_2}{\beta\delta}\int_0^\infty \mu(s) \l \eta_x(s),\psi_t\r\,\d s \leq C\|\psi_t\|\|\eta\|_\M
\leq \rho_2\|\psi_t\|^2 + C\Gamma[\eta]
$$
and, for every $\eps_J>0$ small,
\begin{align*}
&- \frac{2 \rho_3 b}{\delta} \l\theta,\psi_x\r + \frac{2 \rho_3\kappa}{\delta}\l \theta,\varphi + \Psi \r
+ 2\rho_3\|\theta\|^2 \\
&\leq C\big[\|\psi_x\| + \|\varphi_x+\psi\| \big]\|\theta\| + C\|\theta\|^2\\
&\leq  \eps_J \big[ \|\psi_x\|^2 + \|\varphi_x + \psi\|^2 \big] + \frac{C}{\eps_J}\|\theta\|^2,
\end{align*}
with $C$ independent of $\eps_J$, the claim follows.
\end{proof}

\subsection{The functional $\boldsymbol{K}$}
We introduce the number
$$
\gamma_g = \kappa - \frac{g(0)\rho_1}{\beta\rho_3}
$$
depending on the memory kernel $g$.
It is readily seen that
$$
\chi_g = 0 \imply \gamma_g \not=0.
$$
Then, assuming $\chi_g = 0$ and calling
\begin{align*}
K_1(t) &=  \frac{\rho_1 b}{\kappa}\l \psi_x(t), \varphi_t(t) \r + \rho_2 \l \psi_t(t), \varphi_x(t)
+ \psi(t) \r,\\
K_2(t) &=  \int_0^\infty \mu(s) \l \eta_x^t (s), \varphi_x(t) + \psi(t) \r\, \d s,\\
K_3(t) &= - \frac{\delta\rho_1}{\kappa} \l \theta(t), \varphi_t(t) \r,
\end{align*}
we set
$$
K(t) =  \frac{2\kappa}{\gamma_g} \bigg[\frac{\gamma_g}{\kappa} K_1(t)
- \frac{\rho_1\delta}{\beta\rho_3\kappa} K_2(t)
+ K_3(t) \bigg].
$$

\begin{lemma}
\label{Lkappa}
Suppose that $\chi_g=0$. Then $K$ satisfies the differential inequality
$$
\ddt K + \kappa\|\varphi_x+\psi\|^2 \leq c_K \big[\|\psi_t\|^2
+ \Gamma[\eta] \big]
$$
for some $c_K>0$.
\end{lemma}

\begin{proof}
In light of \eqref{Tuno} and \eqref{Tdue}, we obtain the identity
\begin{equation}
\label{A}
\ddt K_1 +\kappa\|\varphi_x + \psi\|^2
= \Big(\rho_2  - \frac{\rho_1 b}{\kappa} \Big)\l \psi_{t}, \varphi_{tx} \r
+ \rho_2\|\psi_t\|^2 - \delta \l \theta_x, \varphi_x + \psi \r.
\end{equation}
By \eqref{Tquattro},
\begin{align*}
\ddt K_2 &= -\int_0^\infty \mu(s) \l T\eta(s), (\varphi_x + \psi)_x \r\, \d s
- g(0) \l \theta , (\varphi_x + \psi)_x \r\\ \nonumber
 &\quad - \int_0^\infty \mu(s) \l \eta_{xx}(s) , \varphi_{t} \r\, \d s + \int_0^\infty \mu(s) \l \eta_x (s), \psi_t \r\, \d s .
\end{align*} From \eqref{Ttre} we learn that
$$
- \int_0^\infty \mu(s) \l \eta_{xx}(s) , \varphi_{t} \r\, \d s =
- \beta\rho_3\l \theta_t,\varphi_t \r + \beta \delta\l \psi_{t},\varphi_{tx} \r,
$$
while an integration by parts in $s$ yields
(again, the boundary terms vanish)
$$
-\int_0^\infty \mu(s) \l T\eta(s), (\varphi_x + \psi)_x \r\, \d s = \int_0^\infty \mu'(s) \l \eta_x(s), \varphi_x + \psi \r\, \d s.
$$
We conclude that
\begin{align}
\label{B}
\ddt K_2
& =  \int_0^\infty \mu'(s) \l \eta_x(s), \varphi_x + \psi \r\, \d s  + \int_0^\infty \mu(s) \l \eta_x (s), \psi_t \r\, \d s\\\noalign{\vskip1.7mm}
\nonumber &\quad - \beta\rho_3\l \theta_t,\varphi_t \r + \beta \delta\l \psi_{t},\varphi_{tx} \r + g(0) \l \theta_x , \varphi_x + \psi \r.
\end{align}
Finally, exploiting once more \eqref{Tuno},
\begin{equation}
\label{C}
\ddt K_3 = -\frac{\delta\rho_1}{\kappa}\l \theta_t,\varphi_t\r + \delta \l \theta_x,\varphi_x + \psi \r.
\end{equation}
At this point, reconstructing $K$ from \eqref{A}-\eqref{C}, we are led to the differential identity
\begin{align*}
&\ddt K + 2\kappa\|\varphi_x+\psi\|^2\\
& = \chi_g \frac{2g(0)\kappa b}{\beta\gamma_g}\l \psi_{t}, \varphi_{tx} \r + 2\rho_2\|\psi_t\|^2 \\
&\quad- \frac{2\rho_1 \delta}{\beta\gamma_g\rho_3}\bigg[ \int_0^\infty \mu'(s) \l \eta_x(s), \varphi_x + \psi \r\, \d s
+ \int_0^\infty \mu(s) \l \eta_x (s), \psi_t \r\, \d s \bigg].
\end{align*}
Since $\chi_g = 0$ by assumption, we are left to control the integral terms
in the right-hand side.
We have
\begin{align*}
- \frac{2\rho_1 \delta}{\beta\gamma_g\rho_3} \int_0^\infty \mu'(s) \l \eta_x(s), \varphi_x + \psi \r\, \d s
&\leq C\|\varphi_x + \psi\| \int_0^\infty -\mu'(s)\|\eta_x(s)\|\,\d s\\
&\leq C\|\varphi_x + \psi\|\sqrt{\Gamma[\eta]}\\
\noalign{\vskip1.5mm}
& \leq \kappa\|\varphi_x + \psi\|^2 + C\Gamma[\eta],
\end{align*}
and, recalling \eqref{NormaETA},
$$
- \frac{2\rho_1 \delta}{\beta\gamma_g\rho_3} \int_0^\infty \mu(s) \l \eta_x (s), \psi_t \r\, \d s
\leq C\|\psi_t\|\|\eta\|_\M \leq C \|\psi_t\|^2 + C\Gamma[\eta].
$$
The proof is completed.
\end{proof}

\subsection{The functional $\boldsymbol{L}$}
Let
$$
L(t) = 2\rho_2 \l \psi_t(t),\psi(t) \r -2\rho_1 \l \varphi_t(t),\varphi(t)\r.
$$

\begin{lemma}
\label{Lelle}
The functional $L$ satisfies the differential inequality
$$
\ddt L + \rho_1\|\varphi_t\|^2 + b\|\psi_x\|^2 \leq c_L\big[\|\varphi_x+\psi\|^2
+\|\psi_t\|^2 + \|\theta\|^2 \big]
$$
for some $c_L>0$.
\end{lemma}

\begin{proof}
By means of \eqref{Tuno} and \eqref{Tdue},
$$
\ddt L + 2\rho_1\|\varphi_t\|^2 + 2b\|\psi_x\|^2= 2\rho_2\|\psi_t\|^2 
+ 2\delta \l \theta,\psi_x\r + 2\kappa\|\varphi_x+\psi\|^2 -4\kappa\l \varphi_x+\psi,\psi\r.
$$
Since the right-hand side is easily controlled by
$$
b\|\psi_x\|^2 + C\big[\|\varphi_x+\psi\|^2
+\|\psi_t\|^2 + \|\theta\|^2\big],
$$
we are done.
\end{proof}

\section{Proof of Theorem \ref{MAIN} (Sufficiency)}
\label{Suf}

\noindent
Within the condition $\chi_g = 0$, we are now in the position to prove the exponential stability of $S(t)$.
In what follows, $E$ is (twice) the energy,
whereas $I,J,K,L$ denote the functionals of
the previous Section \ref{Aux}.

\subsection{A further energy functional}
For $\eps>0$, we define
$$
M_\eps(t) = I(t) + \eps J(t) + \frac{\eps\rho_2}{2 c_K} K(t) + \eps\sqrt{\eps}\, L(t),
$$
where $c_K>0$ is the constant of Lemma \ref{Lkappa}.

\begin{lemma}
\label{Lemme}
For every $\eps>0$ sufficiently small, the differential inequality
$$
\ddt M_\eps + \eps^2 E
\leq \frac{c_M}{\eps} \Gamma[\eta]
$$
holds for some $c_M>0$ independent of $\eps$.
\end{lemma}

\begin{proof}
Collecting the inequalities of Lemmas \ref{Li}, \ref{Lgei}, \ref{Lkappa} and \ref{Lelle},
we end up with
\begin{align*}
&\ddt M_\eps+\eps\Big(\frac{\kappa\rho_2}{2c_K}-\eps_J-\sqrt{\eps}\, c_L\Big)\|\varphi_x+\psi\|^2
+\eps\sqrt{\eps}\,\rho_1\|\varphi_t\|^2+\eps(\sqrt{\eps}\,b-\eps_J)\|\psi_x\|^2\\
&\quad
+\Big(\frac{\eps\rho_2}{2}-\eps_I-\eps\sqrt{\eps}\, c_L\Big)\|\psi_t\|^2
+\Big(\rho_3-\eps\sqrt{\eps}\, c_L-\frac{\eps c_J}{\eps_J}\Big)\|\theta\|^2+\|\eta\|^2_\M\\
&\leq \Big(\frac{c_I}{\eps_I}+\frac{\eps c_J}{\eps_J}+\frac{\eps\rho_2}{2}\Big)\Gamma[\eta].
\end{align*}
At this point, we choose
$$
\eps_I = \frac{\eps\rho_2}{4} \and \eps_J = \frac{2\eps c_J}{\rho_3}.
$$
Taking $\eps>0$ sufficiently small, the claim follows.
\end{proof}

\subsection{Conclusion of the proof of Theorem \ref{MAIN}}
By virtue of \eqref{Eidentity} and Lemma \ref{Lemme},
for $\eps>0$ sufficiently small the functional
$$
G_\eps(t) = E(t) + \eps^2 M_\eps(t)
$$
fulfills the differential inequality
$$
\ddt G_\eps + \eps^4 E\leq -\Big(\frac{1}{\beta}-c_M\eps\Big)\Gamma[\eta]\leq 0.
$$
It is also clear from the definition of the functionals involved that,
for all $\eps>0$ small,
$$\frac12 E(t)\leq G_\eps(t)\leq 2E(t).$$
Therefore, an application of the Gronwall lemma entails the required exponential decay
of the energy.
\qed

\begin{remark}
It is worth observing that, contrary to what done in \cite{SJR}, here the proof of exponential stability
is based on the construction of explicit energy-like functionals.
The advantage (with respect to linear semigroups techniques) is that the same calculations
apply to the analysis of nonlinear version of the problem, allowing, for instance, to prove the
existence of absorbing sets.
\end{remark}

\section{Proof of Theorem \ref{MAIN} (Necessity)}
\label{Nec}

\noindent
In this section, we show that the semigroup $S(t)$ is not exponentially stable when the stability number
$\chi_g$ is different from zero.
The proof is based on the following abstract result from~\cite{Pru} (see also~\cite{GNP} for the statement used here).

\begin{lemma}
\label{pruss}
$S(t)$ is exponentially stable
if and only if there exists $\eps>0$ such that
\begin{equation}
\label{cns}
\inf_{\lambda\in \R}\|\i\lambda z-\A z\|_{\H}\geq \eps\|z\|_{\H},\quad \forall z \in \D(\A),
\end{equation}
where $\A$ and $\H$ are understood to be the complexifications
of the original infinitesimal generator and
phase space, respectively.
\end{lemma}

The strategy consists in verifying that condition \eqref{cns} fails to hold.
Without loss of generality, we can take $\ell = \pi$.
Accordingly, for every $n\in\N$, the vector
$$
\textstyle \zeta_n = \big(0,\frac{\sin{nx}}{\rho_1},0,0,0,0\big)
$$
satisfies
$$
\textstyle\|\zeta_n\|_\H = \sqrt{\frac{\pi}{2 \rho_1}}\,.
$$
For all $n\in\N$,
we denote for short
$$
\textstyle \lambda_n = \sqrt{\frac{\kappa}{\rho_1}}\, n
$$
and we study the equation
$$
\i\lambda_n z_n-\A z_n=\zeta_n
$$
in the unknown variable
$$z_n =(\varphi_n,\tilde\varphi_n,\psi_n,\tilde\psi_n,\theta_n,\eta_n).$$
Our conclusion is reached if we show that $z_n$ is not bounded in $\H$,
since this would violate \eqref{cns}.
Straightforward calculations entail the system
$$
\begin{cases}
\rho_1 \lambda_n^2 \varphi_n + \kappa (\varphi_{nx}+\psi_n)_x= -\sin{nx},\\
\noalign{\vskip1.5mm}
\rho_2\lambda_n^2 \psi_n + b \psi_{nxx} - \kappa (\varphi_{nx}+\psi_n) - \delta \theta_{nx} = 0,\\
\noalign{\vskip1.5mm}
\displaystyle\i\rho_3\lambda_n\theta_n - \frac{1}{\beta}\int_0^\infty \mu(s)\eta_{nxx}(s)\,\d s
+ \i \delta \lambda_n \psi_{nx} =0,\\
\noalign{\vskip1mm}
\i\lambda_n\eta_n -T\eta_n - \theta_n =0.
\end{cases}
$$
We now look for solutions (compatible with the boundary conditions) of the form
\begin{align*}
\varphi_n &= A_n \sin{nx},\\
\psi_n &= B_n \cos{nx},\\
\theta_n &= C_n \sin{nx},\\
\eta_n &= \phi_n(s) \sin{nx},
\end{align*}
for some $A_n,B_n,C_n\in\mathbb{C}$ and some complex square summable function
$\phi_n$ on $\R^+$ with respect to the measure $\mu(s)\d s$,
satisfying $\phi_n(0)=0$. 
This yields
$$
\begin{cases}
\kappa n B_n = 1,\\
\noalign{\vskip1.5mm}
k n A_n + \big(-\rho_2\lambda_n^2  + b n^2 + \kappa\big) B_n + \delta n C_n = 0,\\
\noalign{\vskip1.5mm}
\displaystyle\i\rho_3\lambda_n C_n
+ \frac{n^2}{\beta}\int_0^\infty \mu(s)\phi_{n}(s)\,\d s - \i \delta \lambda_n n B_n =0,\\
\noalign{\vskip1mm}
\i\lambda_n\phi_n +\phi'_n - C_n =0.
\end{cases}
$$
An integration of the last equation gives
$$
\phi_n(s) = \frac{C_n}{\i\lambda_n} \big(1-\e^{-\i\lambda_n s} \big).
$$
Substituting the result into the third equation above, and denoting by
$$
\hat\mu(\lambda_n)=\int_0^\infty \mu(s)\e^{-\i\lambda_n s}\,\d s
$$
the Fourier transform\footnote{Since $\mu$ is continuous nonincreasing and summable,
it is easy to see that $\hat\mu(\lambda_n)\neq 0$ for every $n$.} of $\mu$, we find the
explicit solution
$$
A_n = \frac{\rho_2\kappa n^2 - \rho_1 b n^2 -\rho_1\kappa}{\rho_1 \kappa^2 n^2}
-\frac{\delta^2\beta}{\rho_3\kappa\beta\gamma_g + \rho_1 \kappa \hat\mu(\lambda_n)},
$$
where, according to the notation of Section \ref{Aux},
$$
\gamma_g = \kappa - \frac{g(0)\rho_1}{\beta\rho_3}.
$$
At this point, we consider separately two cases.

\subsection*{Case $\boldsymbol{\gamma_g = 0}$}
We have
$$
A_n = \frac{\rho_2\kappa n^2 - \rho_1 b n^2 -\rho_1\kappa}{\rho_1 \kappa^2 n^2}
-\frac{\delta^2\beta}{\rho_1 \kappa \hat\mu(\lambda_n)}.
$$
Due to the convergence $\hat\mu(\lambda_n)\to 0$, ensured by the Riemann-Lebesgue lemma, we find the asymptotic expression
as $n\to\infty$
$$
A_n \sim -\frac{\delta^2\beta}{\rho_1 \kappa \hat\mu(\lambda_n)}.
$$
Since
$$\|z_n\|^2_\H\geq \kappa\|\varphi_{nx}+\psi_n\|^2
+b\|\psi_{nx}\|^2,
$$
there exists $\varpi>0$ such that
$$
\|z_n\|_\H \geq \varpi\|\varphi_{nx}\|=\varpi n|A_n|\bigg(\int_0^\pi \cos^2 nx\,\d x\bigg)^\frac12
=\frac{\varpi \sqrt{\pi}}{\sqrt{2}}n|A_n| \to \infty.
$$

\subsection*{Case $\boldsymbol{\gamma_g \not= 0}$}
Exploiting again the Riemann-Lebesgue lemma, we now get
$$
A_n \to \frac{1}{\kappa}\bigg(\frac{\rho_2}{\rho_1} - \frac{b}{\kappa}\bigg)
-\frac{\delta^2}{\rho_3\kappa\gamma_g }
=\frac{\rho_3 b g(0)}{\rho_1\rho_3\beta \gamma_g}\,\chi_g\neq 0,
$$
as $\chi_g\neq 0$ by assumption. As before, we end up with
$$
\|z_n\|_\H \geq \frac{\varpi \sqrt{\pi}}{\sqrt{2}}n|A_n| \to \infty.
$$
This finishes the proof.
\qed

\begin{remark}
The proof above actually holds within the same minimal assumptions on the memory kernel ensuring the existence
of $S(t)$, i.e.\ $\mu$ nonnull, nonnegative, nonincreasing,
absolutely continuous and
summable on $\R^+$.
\end{remark}

\section{More on the Comparison with the Cattaneo Model}
\label{GoodComp}

\noindent
As previously observed in Section~\ref{Comparison}, at a formal level it is possible to recover
the exponential stability (as well as the lack of exponential stability) of the
Timoshenko-Cattaneo system from our main Theorem~\ref{MAIN}.
Here, we give a rigorous proof of this fact. To this aim, let us
write explicitly the system studied in \cite{SR,SJR}
$$
\begin{cases}
\rho_1 \varphi_{tt} -\kappa(\varphi_x +\psi)_x = 0,\\
\rho_2 \psi_{tt} -b\psi_{xx} +\kappa(\varphi_x +\psi) +\delta\theta_x= 0,\\
\rho_3 \theta_t + q_x +\delta\psi_{tx}=0,\\
\tau q_t + \beta q + \theta_x = 0,
\end{cases}
$$
which generates a contraction semigroup $\hat S(t)$ acting on the phase space
$$
\hat\H = H_0^1({\mathfrak I})\times L^2({\mathfrak I})\times H^1_*({\mathfrak I}) \times L^2_*({\mathfrak I}) \times L^2({\mathfrak I}) \times L^2({\mathfrak I})
$$
normed by
$$
\|( \varphi,\tilde\varphi,\psi,\tilde\psi,\theta,q )\|_{\hat \H}^2 = \kappa\|\varphi_x+\psi\|^2 +\rho_1\|\tilde\varphi\|^2
+b\|\psi_x\|^2+\rho_2\|\tilde\psi\|^2 + \rho_3\|\theta\|^2 + \tau\|q\|^2.
$$
According to the article \cite{SJR}, the semigroup
$\hat S(t)$ is exponentially stable if and only if the stability number
$\chi_\tau$ equals zero.

\smallskip
Along with $\hat S(t)$, we consider the semigroup $S(t)$ on $\H$ generated by \eqref{EQLIN}
for the particular choice of the kernel
$$
\mu(s)=-g_\tau'(s)=\frac{\beta^2}{\tau^2} \,\e^{-\frac{s \beta}{\tau}}.
$$
Then, we define the map $\Lambda:\M\to L^2({\mathfrak I})$ as
$$\Lambda\eta=- \frac{1}{\beta}\int_0^\infty \mu(s)\eta_x(s)\,\d s.
$$
On account of the H\"older inequality,
\begin{equation}
\label{trerry}
\tau\|\Lambda\eta\|^2 \leq \frac{\tau}{\beta^2}\bigg[\int_0^\infty \mu(s)\|\eta_x(s)\|\,\d s\bigg]^2
\leq \frac{1}{\beta}\|\eta\|_\M^2.
\end{equation}
Due to the peculiar form of the kernel, the following result is a direct consequence of the equations.
The easy proof is left to the reader.

\begin{lemma}
\label{iddu}
Let $z_0=(u_0,\eta_0)\in\H$ be any initial datum, where $u_0$ subsumes the first 5 components of $z_0$,
and call $\hat z_0=(u_0,\Lambda\eta_0)\in\hat\H$. Then
the first 5 components of $S(t)z_0$ and $\hat S(t)\hat z_0$ coincide.
Besides, the last component $q(t)$ of $\hat S(t)\hat z_0$ fulfills
the equality
$$q(t)=\Lambda\eta^t,$$
where $\eta^t$ is the last component of $S(t)z_0$.
\end{lemma}

The full equivalence between the two models is established in the next two propositions.

\begin{proposition}
If $S(t)$ is exponentially stable on $\H$, then so is $\hat S(t)$ on $\hat \H$.
\end{proposition}

\begin{proof}
Let $\hat z_0=(u_0,q_0)\in\hat\H$ be fixed.
Choosing $\eta_0\in\M$ of the form
$$
\eta_0(x,s) = -\tau \int_0^x q_0(y)\,\d y,
$$
it is readily seen that $\Lambda\eta_0=q_0$ and
$$
\frac{1}{\beta}\|\eta_0\|_\M^2 = \frac{\tau^2}{\beta} \|q_0\|^2 \int_0^\infty \mu(s)\,\d s = \tau\|q_0\|^2.
$$
Lemma \ref{iddu} yields the identity
$$\|\hat S(t)\hat z_0\|_{\hat \H}
=\|\hat S(t)(u_0,\Lambda\eta_0)\|_{\hat \H}=\|(u(t),\Lambda\eta^t)\|_{\hat \H},
$$
where $u(t)$ denotes the first 5 components of either solution.
On the other hand, we infer from~\eqref{trerry} and the exponential stability of $S(t)$ that
$$\|(u(t),\Lambda\eta^t)\|_{\hat \H}\leq
\|(u(t),\eta^t)\|_{\H}\leq C\|(u_0,\eta_0)\|_\H\e^{-\omega t},
$$
for some $\omega>0$ and $C\geq 1$.
Since
$$\|(u_0,\eta_0)\|_\H=\|\hat z_0\|_{\hat \H},$$
we are finished.
\end{proof}

\begin{proposition}
If $\hat S(t)$ is exponentially stable on $\hat \H$, then so is $S(t)$ on $\H$.
\end{proposition}

\begin{proof}
To simplify the notation, we introduce the 5-component space
$$
\V = H_0^1({\mathfrak I})\times L^2({\mathfrak I})\times H^1_*({\mathfrak I}) \times L^2_*({\mathfrak I}) \times L^2({\mathfrak I})
$$
normed by
$$
\|( \varphi,\tilde\varphi,\psi,\tilde\psi,\theta)\|_{\V}^2 = \kappa\|\varphi_x+\psi\|^2 +\rho_1\|\tilde\varphi\|^2
+b\|\psi_x\|^2+\rho_2\|\tilde\psi\|^2 + \rho_3\|\theta\|^2.
$$
For a fixed $z_0=(u_0,\eta_0)\in\H$, we set
$$
\hat z_0 = (u_0,\Lambda\eta_0)\in\hat\H.
$$
The exponential stability of $\hat S(t)$ and~\eqref{trerry}
imply that
$$\|u(t)\|_\V^2\leq \|\hat S(t)\hat z_0\|_{\hat \H}^2
\leq C\e^{-\omega t}\|\hat z_0\|_{\hat \H}^2
\leq C\e^{-\omega t}\|z_0\|_{\H}^2,$$
for some $\omega>0$ and $C\geq 1$.
Again, on account of Lemma~\ref{iddu},
$u(t)$ denotes the first 5 components of either solution.
Thus,
exploiting the energy identity~\eqref{Eidentity} together with~\eqref{NormaETA},
we arrive at the differential inequality
$$\ddt \|S(t)z_0\|_{\H}^2+\nu \|S(t)z_0\|_{\H}^2\leq \nu \|u(t)\|_\V^2\leq C \nu\e^{-\omega t}\|z_0\|_{\H}^2.
$$
A standard application of
the Gronwall entails the sought exponential decay estimate.
\end{proof}



\end{document}